\documentclass{article}
\usepackage[utf8]{inputenc}

\usepackage{amsmath,amssymb,amsthm,amscd,textcomp,makeidx}
\usepackage{tikz}
\usepackage{standalone} 
\usepackage{tikzscale} 
\usetikzlibrary{backgrounds,fit,positioning}

\usepackage{hyperref}
\hypersetup{pdftex,colorlinks=true,allcolors=blue}
\usepackage{hypcap}
\usepackage{todonotes}
\usepackage{natbib}

\usepackage{authblk}

\DeclareMathOperator{\crit}{crit}
\DeclareMathOperator{\rng}{rng}
\newcommand{\rnglpo}{\operatorname{rng^+}}

\newcommand{\al}{\alpha}
\newcommand{\be}{\beta}
\newcommand{\ga}{\gamma}
\newcommand{\ka}{\kappa}
\newcommand{\calL}{\mathcal{L}}
\newcommand{\from}{\colon}
\newcommand{\of}{\circ}
\newcommand{\restr}{\upharpoonright}
\newcommand{\st}{\,|\,}
\newcommand{\sat}{\vDash}
\newcommand{\calA}{\mathcal{A}}
\newcommand{\calP}{\mathcal{P}}

\newcommand{\starint}[1]{\overset{*}{\cap}V_{#1}}

\author[a]{Andrew D. Brooke-Taylor}
\author[b]{Scott Cramer}
\author[c]{Sheila K. Miller Edwards}
\affil[a]{School of Mathematics\\ University of Leeds\\ Leeds LS2 9JT\\ United Kingdom \\\href{mailto:a.d.brooke-taylor@leeds.ac.uk}{a.d.brooke-taylor@leeds.ac.uk}}
\affil[b]{Department of Mathematics\\California State University, San Bernardino\\5500 University Parkway, San Bernardino, CA 92407\\United States of America\\\href{mailto:scott.cramer@csusb.edu}{scott.cramer@csusb.edu}}
\affil[c]{School of Mathematical and Natural Sciences\\ New College of Interdisciplinary Sciences\\ Arizona State University\\ 4701 W Thunderbird Rd, Glendale AZ, 85306\\ United States of America\\\href{mailto:sheila.miller@asu.edu}{sheila.miller@asu.edu}}
\title{A free two-generated left distributive algebra of elementary embeddings}
\date{}

\begin{document}

\newtheorem{thm}{Theorem}
\newtheorem{cor}[thm]{Corollary}
\newtheorem{lem}[thm]{Lemma}
\newtheorem{prop}[thm]{Proposition}
\newtheorem{conj}[thm]{Conjecture}
\newtheorem{clm}[thm]{Claim}
\newtheorem{fact}[thm]{Fact}
\newtheorem{qtn}[thm]{Question}

\theoremstyle{remark}
\newtheorem{rmk}[thm]{Remark}

\theoremstyle{definition}
\newtheorem{dfn}[thm]{Definition}

\newcommand{\la}{\left <}
\newcommand{\ra}{\right >}
\newcommand{\Acal}{\mathcal{A}}
\newcommand{\Ecal}{\mathcal{E}}
\newcommand{\Dcal}{\mathcal{D}}
\newcommand{\La}{\mathcal{L}}
\newcommand{\M}{\mathcal{M}}
\newcommand{\Bcal}{\mathcal{B}}
\newcommand{\Pcal}{\mathcal{P}}
\newcommand{\Rcal}{\mathcal{R}}
\newcommand{\Wcal}{\mathcal{W}}
\newcommand{\K}{\mathcal{K}}
\newcommand{\Mo}{\mathbb{M}}
\newcommand{\Scal}{\mathcal{S}}
\newcommand{\uphp}{\upharpoonright}
\newcommand{\ov}{\overline}
\newcommand{\aph}{{\aleph_0}}
\newcommand{\R}{\mathbb{R}}
\newcommand{\sub}{\subseteq}
\newcommand{\cof}{\text{cof}}
\newcommand{\AD}{\text{AD}}
\newcommand{\ADR}{\text{AD}_\mathbb{R}}
\newcommand{\ADP}{\text{AD}^+}
\newcommand{\HOD}{\text{HOD}}
\newcommand{\OD}{\text{OD}}
\newcommand{\ot}{\text{ot}}
\newcommand{\Ord}{\text{Ord}}
\newcommand{\ZFC}{\textup{ZFC}}
\newcommand{\Hom}{\text{Hom}}
\newcommand{\rest}{\upharpoonright}
\newcommand{\Ult}{\text{Ult}}
\newcommand{\DC}{\text{DC}}
\newcommand{\HOM}{\text{Hom}}
\newcommand{\cspace}{{}^\omega 2}
\newcommand{\GCH}{\text{GCH}}
\newcommand{\CH}{\text{CH}}
\newcommand{\Prod}{\prod}
\newcommand{\reals}{\R}
\newcommand{\forces}{\Vdash}
\newcommand{\Coll}{\text{Coll}}

\newcommand{\lvl}{L(V_{\lambda+1})}
\newcommand{\vl}{(V_{\lambda+1})}
\newcommand{\vlm}{{V_{\lambda+1}}}
\newcommand{\vlb}{(V_{\bar \lambda+1})}
\newcommand{\vlmb}{{V_{\bar \lambda+1}}}
\newcommand{\vlmbj}{{V_{\bar \lambda_J+1}}}

\newcommand{\blam}{{\bar \lambda}}
\newcommand{\hJ}{{\pi_J}}
\newcommand{\hK}{{\pi_K}}
\newcommand{\hL}{{\pi_L}}
\newcommand{\hhJ}{{\hat J}}
\newcommand{\hhK}{{\hat K}}
\newcommand{\hhL}{{\hat L}}
\newcommand{\Jt}{{(J, \la j_i \ra)}}
\newcommand{\Kt}{{(K, \la k_i \ra)}}
\newcommand{\Lt}{{(L, \la \ell_i \ra)}}
\newcommand{\vbbeta}{{\vec{\bar{\beta}}}}
\newcommand{\bbeta}{{\eta'}}
\newcommand{\balpha}{{\eta}}
\newcommand{\rhoa}{{\rho_\alpha}}
\newcommand{\rhoba}{{\rho_\balpha}}
\newcommand{\mvar}{{\xi}}
\newcommand{\svar}{{\alpha}}
\newcommand{\jseq}{{\la j_i \ra}}
\newcommand{\jpseq}{{\la j_i' \ra}}
\newcommand{\kseq}{{\la k_i \ra}}
\newcommand{\kpseq}{{\la k_i' \ra}}
\newcommand{\jbseq}{{\la j_i \rest L_\beta \vl \ra}}
\newcommand{\kbseq}{{\la k_i \rest L_\beta \vl \ra}}

\newcommand{\fcal}{\mathcal{F}}
\newcommand{\Mcal}{\mathcal{M}}
\newcommand{\mcal}{\mathcal{M}}
\newcommand{\Fcal}{\mathcal{F}}
\newcommand{\fix}{\text{Fix}}
\newcommand{\fpf}{\text{Fpf}}

\newcommand{\gothm}{{\mathfrak{M}}}
\newcommand{\mfrak}{{\mathfrak{M}}}

\newcommand{\Q}{\mathbb{Q}}
\newcommand{\Pbb}{\mathbb{P}}
\newcommand{\Bbbb}{\mathbb{B}}

\newcommand{\Emb}{\mathcal{E}}
\newcommand{\EMB}{{\text{EMB}}}
\newcommand{\Il}{{\text{InvL}}}
\newcommand{\IL}{{\text{INVL}}}
\newcommand{\invlim}{{InvLim}}
\newcommand{\dom}{\text{dom}}
\newcommand{\Norm}{\text{Norm}}
\newcommand{\Fix}{\text{Fix}}

\newcommand{\Jext}{J^{\text{ext}}}
\newcommand{\Jcext}{J^{\text{cext}}}
\newcommand{\Kext}{K^{\text{ext}}}
\newcommand{\Kcext}{K^{\text{cext}}}
\newcommand{\erng}{{\text{erng}}}
\newcommand{\ext}{\text{ext}}

\newcommand{\lb}{\left |}
\newcommand{\rb}{\right |}

\newcommand{\Hull}{\text{Hull}}
\newcommand{\Def}{\text{Def}}
\newcommand{\sseq}{\subseteq}
\newcommand{\rank}{\text{rank}}

\newcommand{\Kcal}{\mathcal{K}}
\newcommand{\LR}{L(\mathbb{R})}
\newcommand{\Vcal}{\mathcal{V}}
\newcommand{\Tcal}{\mathcal{T}}

\newcommand{\len}{\text{len}}
\newcommand{\Aomegae}{\Acal_\omega^{\text{emb}}}
\newcommand{\irng}{\text{irng}}
\newcommand{\id}{\text{id}}

\newcommand{\Jv}{(J, \vec j)}
\newcommand{\Kv}{(K, \vec k)}
\newcommand{\otp}{\text{otp}}
\newcommand{\hqod}{\text{HqOD}}
\newcommand{\jhqod}{\text{$j$-HqOD}}

\maketitle
\newpage
\begin{abstract}
The set-theoretic large cardinal axiom known as I3 posits the existence of a non-trivial \emph{rank-to-rank embedding} from an 
initial segment of the universe of sets into itself.
Laver \cite{Laver:1992} showed that the algebra
generated by a single such 
embedding under the operation of \emph{application} is in fact the free left distributive 
(LD) algebra on one generator.  This and associated theorems using the set-theoretic structure
of the embeddings yielded numerous results about general LD algebras under the
assumption of I3, only some of which 
have since been proven without the use of such a strong axiom. 
A natural question is whether, under the assumption of I3, 
one can obtain a free LD algebra of embeddings
on more than one generator.  
Here we show that, under an assumption only just above I3 in the large 
cardinal hierarchy (namely, I2), we indeed obtain a two-generated free left distributive algebra of rank-to-rank embeddings.
\end{abstract}
\section{Introduction}
Left distributive operations arise in various areas of mathematics, including the theories of knots, braids, and groups, and also large cardinal axioms of set theory. 
A left distributive algebra (LD algebra for short, also sometimes referred to
in the literature as a \emph{shelf} \cite{Crans:2004}) 
is an algebraic structure for which left multiplication by
any element is a homomorphism; 
that is, a left distributive algebra is a set $A$ endowed with a binary operation $\cdot$ such that for all $x$, $y$, and $z$ in $A$, $x \cdot(y \cdot z) = (x \cdot y)\cdot(x \cdot z)$. Familiar examples of left distributive operations include group conjugation and the weighted mean $x \cdot y = rx + (1-r)y$. 
In both of these examples, and many others that arise in the study of knots and braids, left multiplication by an element of the set is in fact an automorphism; structures satisfying this extra requirement were called automorphic sets by Brieskorn \cite{Brieskorn:1988} and in the ensuing set-theoretic literature,
but are commonly referred to as racks in the algebra literature
after Fenn and Rourke \cite{fenn1992racks} (who attribute the terminology to Conway and Wraith). 
Idempotent automorphic sets 
(those in which $x \cdot x = x$ for every element $x$) 
were named quandles by Joyce \cite{Joyce:1982}.

Beginning in the 1980s, a significant relationship was discovered between left distributive algebras and certain large cardinal axioms in set theory. A large cardinal axiom is one that implies the consistency of the usual axioms of set theory (the Zermelo-Frankel axioms, plus the Axiom of Choice: ZFC), and therefore–––by G\"odel's incompleteness theorems–––cannot be derived from ZFC. The large cardinal axioms themselves are (essentially)
linearly ordered in consistency strength, and the large cardinals that give rise to left distributivity are among the largest not known to be inconsistent. 

Large cardinals in the upper regions of the large cardinal hierarchy are
generally characterized by the existence of a nontrivial 
\emph{elementary embedding} $j$ from the universe of sets $V$ to a transitive 
model $M$ of the ZFC axioms. 
A first order language $\calL$ may consist of symbols for 
constants, operations and relations; 
an $\calL$-structure is then a set endowed with
choices of constants, operations and relations interpreting those symbols,
and an elementary embedding from one $\calL$-structure to another 
is a function
preserving \emph{all} first-order formulas for the language $\calL$
(that is, all formulas built up as expected from the symbols in $\calL$,
logical connectives, and quantifiers, with variables ranging over \emph{elements}
of the structures).
Unless otherwise stated, the language assumed for elementary embeddings 
in the large cardinals setting
is the \emph{language of set theory} $\{\in\}$, containing a single binary
relation symbol $\in$ interpreted as the membership relation;
other familiar set-theoretic constants, operations and relations 
like $\emptyset$, $\cap$ and $\subset$ 
must also be preserved by elementary embeddings
as they are definable in the language of set theory.

Richard Laver proved that certain large cardinal embeddings give rise to a 
\emph{free} left distributive algebra in a natural way \cite{Laver:1992}. 
Namely, one of the strongest large cardinal axioms not known to be inconsistent
posits the existence of elementary embeddings $j\from V_\lambda\to V_\lambda$,
referred to as \emph{rank-to-rank embeddings} or as \emph{I3 embeddings}.
There is a natural
\emph{application} operation definable on pairs of such embeddings,
and this operation is automatically left distributive, by elementarity.

Laver demonstrated the existence of a normal form for elements of a conservative extension of the free one-generated left distributive algebra
$\mathcal{A}$, 
thereby showing that the existence of a nontrivial rank-to-rank embedding implies that the left subterm relation $<_L$ linearly orders $\mathcal{A}$, 
that the word problem for $\mathcal{A}$ is decidable, and that 
the left distributive algebra generated by a single nontrivial 
rank-to-rank embedding under the application operation is in fact
free (see Section \ref{LCintro}). He later proved \cite{Laver:1997}, again under the assumption of the existence of a rank-to-rank embedding, that a certain family of finite left distributive algebras now known as \emph{Laver tables} approximate the free left distributive algebra $\mathcal{A}$. 
Laver's proof of linearity employs the large cardinal assumption to secure the irreflexivity of $<_L$ and uses the normal form to demonstrate that for all $a$ and $b$ in $\mathcal{A}$, either $a \leq_L b$ or $b \leq_L a$. 

Patrick Dehornoy \cite{Dehornoy:1994} later proved the linearity of $<_L$ and the solvability of the word problem within ZFC; he defined an operation $[ \ \ ]$ on an extension of Artin's braid group to infinitely many strands, $B_{\infty}$, and showed that the algebra generated by the closure of a single element of $B_{\infty}$ under $[ \ \ ]$ is isomorphic to $\mathcal{A}$. David Larue, also working in $B_{\infty}$, simplified Dehornoy's proof considerably \cite{Larue:1994}. 
In his PhD thesis Larue built a representation of the free left distributive algebra on $n$ generators (for $n$ a natural number) \cite{LarueThesis:1994}. That expansion of Artin's braid group on infinitely many generators is now known as Larue's group.

A similar example of a free left distributive algebra on more than one generator within large cardinal theory remained elusive. Laver's result that the existence of a rank-to-rank embedding gives rise to the free left distributive algebra on a single generator was announced in 1986 \cite{laver1986elementary}, but since then it has remained unclear whether one can construct a provably free left distributive algebra of embeddings on more than one generator assuming only the existence of a single rank-to-rank embedding from $V_{\lambda}$ to $V_{\lambda}$. 

Dimonte \cite{Dimonte:2019} recently demonstrated that an extremely strong
large cardinal axiom extending rank-to-rank embeddings can give rise to a left distributive algebra with two 
generators with very different properties, neither of which generates the other,
but whether Dimonte's algebra is free remains open.

The main result of this paper is an affirmative answer to the question from a large
cardinal much closer to I3: we construct a free left distributive algebra of embeddings on two generators from a single $\Sigma^1_1$-elementary embedding 
from $V_\lambda$ to $V_\lambda$.
The assumption that such an embedding exists is known as I2,
and is the least standard large cardinal assumption stronger than I3;
we discuss it further in 
Sections \ref{LCintro} and \ref{Main}. 
We use the extra strength provided by $\Sigma^1_1$-elementarity in two distinct ways in the proof: once in obtaining 
``square roots,'' used to show that the ranges of the embeddings we construct are eventually disjoint in a strong way, and
again later to show that this disjointness is preserved under application of an embedding.

The essence of the construction goes like this: from the assumption of the existence of a single, nontrivial, $\Sigma^1_1$-elementary rank-to-rank embedding, we use inverse limit techniques to construct $\Sigma^1_1$-elementary embeddings 
$k$ and $\ell$ with ranges that are eventually disjoint. 
In particular $k$ and $\ell$ have no embeddings in common in their ranges, 
and we show furthermore that it is possible to propagate this property to the embeddings in their ranges, showing that their `iterated ranges' are disjoint. 
With this disjointness
we can use 
Dehornoy's quadrichotomy \cite[Chapter V, Proposition 3.1]{Dehornoy:2000}
to show that two embeddings in the closure of $\{k,\ell\}$ under application are equivalent just in case they are equivalent under the left distributive law, 
and hence that the resulting algebra of embeddings 
is isomorphic to the free left distributive algebra on two generators. 

The paper is organized as follows: in Section \ref{FreeLDs} we give further background on free left distributive algebras and state some open questions; in Section \ref{LCintro} we offer requisite background on large cardinal elementary embedding axioms; in Section \ref{Tools} we present the main technical tools used in the proof of the main result; in Section \ref{Main} we prove our main result.
A reader already familiar with elementary embeddings can skip Section \ref{LCintro}, and one so inclined can reference Section \ref{FreeLDs} without reading it in detail. The techniques in Section \ref{Tools} were originally introduced by Richard Laver(\cite{Laver:1997}, \cite{Laver:2001}) and were further developed by Scott Cramer \cite{Cramer:2015}.

\section{Free Left Distributive Algebras}\label{FreeLDs}

In this section we present definitions and results on free left distributive algebras necessary for our main result. To readers interested in a more comprehensive treatment of the background material, we suggest \cite{Dehornoy:2000}--- especially for the connections between braids and self-distributivity---and \cite{LaverMiller:2013} for a survey of basic results and a proof of the existence of a particular normal form in the single generator case.

For any cardinal $\gamma$, one can create a free left distributive term algebra $\mathcal{A}_\gamma$ on $\gamma$ generators and a single (left distributive) binary operation $\cdot$ by forming the collection $A_{\gamma}$ of all terms in the $\gamma$ generators and then considering their equivalence classes modulo the left distributive law (LD). To be explicit, two terms $u$ and $v$ in ${A}_{\gamma}$ are equivalent under the left distributive law, $u \equiv_{LD} v$, if and only if one can start with $u$ and obtain $v$ by applying a series of substitutions of the form $a \cdot (b \cdot c) \leftrightarrow (a \cdot b) \cdot (a \cdot c)$ to subterms. 
For the remainder of the paper we adopt the standard conventions of writing $ab$ for $a \cdot b$ and $a_0a_1a_2 \cdots a_n$ for $(((a_0a_1)a_2) \cdots a_n)$ where $a, b, a_0, \ldots a_n$ are elements of any left distributive algebra. 
We write $\mathcal{A}$ for $\mathcal{A}_1$ and $\Acal_j$ (respectively $\mathcal{A}_{k, \ell}$) for the algebra of embeddings generated as the closure of $j$ (resp. $k$ and $\ell$) under the application
operation (defined in Section \ref{LCintro} below).\footnote{For ease of notation, we do not write $\mathcal{A}_{\{j\}}$ or $\mathcal{A}_{\{k,\ell\}}$. When the index is a cardinal $\kappa$, $\mathcal{A}_{\kappa}$ is the free left distributive algebra on $\kappa$ generators. We use the latter notation only for $\mathcal{A}_2$ and $\mathcal{A}_n$.}

Most left distributive algebras familiar from classical mathematics are idempotent: $g \cdot g = g$ for every element $g$ of the algebra. Because no generator of a free LD algebra can be written as a product, familiar left distributive algebras, like the one induced by conjugation, are not free. 

Laver was the first to establish the existence of a natural, nontrivial, free left distributive algebra \cite{Laver:1992}. Specifically, he proved that for any nontrivial rank-to-rank embedding $j$, the left distributive algebra $\mathcal{A}_j$ created by closing $\{j\}$ under the application operation is free. He did this by showing that the relation $<_L$ is a linear ordering of the free left distributive algebra on one generator, where $a <_L b$ just in case $b = ab_1 \cdots b_k$ for some $b_1, \ldots, b_k$ in $\mathcal{A}$. He thereby also established that the word problem for $\mathcal{A}$ is solvable and that $\mathcal{A}_j \cong \mathcal{A}$. Note however that the existence of such an embedding $j$ is a strong large cardinal axiom.

To prove linearity of $<_L$ one must show irreflexivity (that no element is less than itself in the sense of $<_L$) and connectivity (that for every pair $a$ and $b$ of elements, either $a \leq_L b$ or $b \leq_L a$). Laver \cite{Laver:1992} obtained irreflexivity from the structure on the large cardinal embeddings: for any nontrivial rank-to-rank embedding $k$, $k \neq k k_1 \cdots k_n$ for any nontrivial elementary embeddings $k_1, \ldots, k_n$. Because the relation $<_L$ is irreflexive on the (left distributive) algebra of embeddings, it must also be irreflexive on the free LD algebra. Working in an expanded left distributive algebra $\mathcal{P}$ that includes a second, `composition' operation, Laver showed connectivity by demonstrating the existence of a normal form theorem for the free LD algebra \cite{Laver:1992}, thereby also showing that $\mathcal{A}$ is left cancelative (so $ab = ac$ if and only if $b = c$ and $ab <_L ac$ if and only if $b <_L c$). He later proved the existence of another normal form called the division form that is more useful for applications \cite{Laver:1995}. The expanded algebra $\mathcal{P}$ mimics the structure of elementary embeddings with respect to both application and composition, satisfying the following laws: $(a \circ b) \circ c = a \circ (b \circ c), (a \circ b)c = a(bc), a(b \circ c) = ab \circ ac, a \circ b = ab \circ a$. The first two assert the usual properties of composition and the third says that left multiplication is still a homomorphism in the expanded algebra. The final identity  is frequently called the `braid law' for reasons that will soon become clear.

Working in $\calP$ rather than $\calA$, and in the $\gamma$-generated case working in $\calP_\gamma$ rather than $\calA_\gamma$, 
naturally gives more flexibility but still allows one to prove results about $\calA_\gamma$: Laver showed
that $\calP_\gamma$ is conservative over $\calA_\gamma$ --- 
terms $u$ and $v$ in $A_\gamma$ are equal in $\calP_\gamma$ if and only if they are equal in $\calA_\gamma$ ---
and likewise $u<_L v$ in $\calP_\gamma$ if and only if $u<_L v$ in $\calA_\gamma$ (\cite[Lemma 3]{Laver:1992}; as noted for example in 
\cite[Theorem 1.2]{Laver:1996}, although the proof in \cite{Laver:1992} is presented for the single generator case, the argument goes through for any
number of generators).

There are deep connections between left distributive algebras and Artin's braid groups, which act on direct powers of racks 
and act partially on direct powers of algebras satisfying left cancellation. Nonadjacent generators of the braid group commute, and adjacent generators satisfy a `braiding' relation. More specifically, for $2 \leq N \leq \infty$, the braid group on $N$ strands is generated by $\sigma_1, \sigma_2, \ldots \sigma_i$ for $i \leq N$ subject to the relations $\sigma_i\sigma_j = \sigma_j\sigma_i$ if $|i-j| >1$ and $\sigma_i\sigma_{i+1}\sigma_i = \sigma_{i+1}\sigma_i\sigma_{i+1}$.

Dehornoy, using an extension of Artin's braid group to infinitely many generators, proved the linearity of $<_L$ without the large cardinal assumption, that is, within ZFC \cite{Dehornoy:1994}. In particular he proved that free LD algebras are \emph{confluent}: two terms $u$ and $v$ in $\mathcal{A}_{\gamma}$ are equivalent if and only if they have a common forward LD-expansion. Namely, $u$ and $v$ are equivalent just in case there is a term $w$ with $u \equiv w \equiv v$, and $w$ can be obtained from $u$ and from $v$ by substituting subterms of the form $ab(ac)$ for subterms of the form $a(bc)$.

The ordering $<_L$ is not a linear ordering of free algebras on more than one generator because the generators themselves are incomparable. Say that $a$ and $b$ have a \emph{variable clash} ($a \nsim b$) if and only if for some (possibly empty) $p$, $a \equiv pxa_1 \cdots a_n$ and $b \equiv pyb_1 \cdots b_k$, where $x$ and $y$ are distinct generators (equivalently: there exists a $p$ in $\Pcal_{\gamma}$ such that $px \leq_L a$ and $py \leq_L b$). Note that  the term $p$ might be in $\Pcal_{\gamma}$ and not in $\Acal_{\gamma}$. Crucially for the sequel, Dehornoy proved that confluence still holds in the free left distributive algebra on $\gamma$ generators \cite{Dehornoy:1989}, where now we have a quadrichotomy in place of linearity. In particular, for any two words in $a$ and $b$ in $\mathcal{A}_{\gamma}$, either $a$ is LD-equivalent to $b$, one of $a$ or $b$ is a left subterm of the other ($a <_L b$ or $b <_L a$), or 
$a$ and $b$ 
have a variable clash; futhermore, all of these relations are exhibited by confluence.

\section{Large Cardinal Elementary Embeddings}\label{LCintro}

G\"odel showed in the 1930s that no consistent recursively defined 
axiom system incorporating arithmetic can give a finite proof of its own consistency. However, the consistency of arithmetic is provable from the axioms of set theory, which constitute a
stronger system than arithmetic alone.  Likewise, the consistency of the axioms of ZFC is a consequence of the assumption of still-stronger axioms known as \emph{large cardinal axioms} or simply \emph{large cardinals}. 

One of the ZFC axioms, the Axiom of Infinity, 
implies the consistency of the other axioms: using it one may construct $V_{\omega}$, the set of all hereditarily finite sets, 
which is a model of all the axioms of set theory except the axiom of infinity.
In a similar way, large cardinal axioms 
generally assert the existence of an infinite cardinal $\ka$ with properties that ensure that $\ka$ is very large,
and in particular that $V_\ka$ is a model of ZFC along with weaker large cardinal axioms.
G\"odel initiated a program to understand the logical consequences of these axioms and other axiom systems known to imply the consistency of ZFC. Although G\"odel's hopes of settling the generalized continuum hypothesis in this way were not realized, the study of large cardinal axioms evolved to be an important area of set theory.

The large cardinal axioms form the benchmarking hierarchy of consistency strengths in set theory, where large cardinal axiom $B$ is stronger than large cardinal axiom $A$ if the consistency of $\ZFC + B$ implies that of $\ZFC+A$. The smaller large cardinal axioms are compatible with the claim that the universe of sets is relatively small and contains only ordinals and sets definable from earlier-constructed sets (the axiom $V = L$); such large cardinal axioms tend to be characterized by combinatorial properties. Large cardinals that imply the universe is too large to satisfy $V=L$ are sometimes called `large' large cardinal axioms and can usually be viewed as hypothesizing the existence of a nontrivial elementary embedding from the universe of sets, $V$, into a transitive substructure $M$ of $V$ satisfying specified closure conditions. The large cardinal axiom is the assertion that such a nontrivial elementary embedding exists; any cardinal that is the least cardinal moved by such an embedding is a such-and-such cardinal.  Asserting progressively stronger closure conditions on the target model $M$ gives a sequence of progressively stronger large cardinal axioms.

As will be detailed below, there is an upper bound to the possible closure of the target model $M$; if we wish to keep the 
Axiom of Choice, the target model $M$ cannot be all of $V$ (see Theorem \ref{Kunen}). There are, however, large cardinal axioms described by embeddings for which the domain and target models are the same, namely a rank-initial segment $V_{\lambda}$ of the universe into itself (where $\lambda$ necessarily a limit cardinal of cofinality $\omega$ or the ordinal successor of such a cardinal). It is these embeddings that are the subject of the rest of this paper.

In the remainder of this section we offer a summary of relevant definitions and facts about elementary embeddings. In all that follows, 
$\lambda$ will always denote a limit cardinal of cofinality $\omega$. 

We say that $j$ is an elementary embedding from $\langle V, \in \rangle$ into a transitive substructure $\langle M, \in \rangle$ if for every formula $\varphi(v_1, \ldots v_n)$ in the language of set theory $\{\in\}$ and every tuple $a_1, \ldots a_n$ of elements of $V$, 
$V \models \varphi[a_1, \ldots a_n]$ if and only if $M \models \varphi[j(a_1), \ldots j(a_n)]$. We will consider only nontrivial elementary embeddings $j$: those that are not the identity function. An important attribute of a nontrivial elementary embedding is the first ordinal it moves. Say that $\kappa$ is the \emph{critical point} of $j$ if and only if $\kappa$ is the least ordinal for which $\kappa < j(\kappa)$.  Note that 
by an inductive argument, the image of $\gamma$ is at least $\gamma$ for all ordinals $\gamma$. Thus, $j(\alpha) = \alpha$ for all $\alpha < \kappa$. 

The following well-known result begins the upward climb of the strength of these embeddings.

\begin{prop} Suppose that $j: V \to M$ is an elementary embedding and that $M$ is transitive. Then for $\kappa = \crit(j)$, $V_{\kappa+1}$ is contained as a subset of $ M$. 
\end{prop}

\begin{proof} As $\kappa = \crit(j)$ is the least ordinal moved by $j$, $j(\alpha) = \alpha$ for every $\alpha < \kappa$. Furthermore, we have inductively that for every set $a$ of rank less than $\kappa$, $j(a) = a$. In particular, for every $a \in V_\kappa$, $j(a) = a \in M$, and thus $V_\kappa \sseq M$ and $V$ and $M$ agree on sets of size less than $\kappa$: $V_\kappa^M=V_\kappa$.

To see that $V_{\kappa+1} \sseq M$, we show that every subset $A$ of $V_{\kappa}$ is an element of $M$. Let $A \sseq V_{\kappa}$ (so $A \in V_{\kappa+1}$). By the above and elementarity, we have that $j(A) \sseq V_{j(\kappa)}^M$ and $j(A) \cap V_{\kappa}^M = j(A) \cap V_\kappa$. For every $a \in A \sseq V_{\kappa}$, we have $j(a) = a$, so elementarity gives that $a \in A$ if and only if $j(a) = a \in j(A)$. Therefore $j(A) \cap V_\kappa = A \cap V_{\kappa} = A$. Hence $A = j(A) \cap V_\kappa^M$ is itself in $M$.  
\end{proof}

We can go no farther for free. Suppose that $U$ is a $\kappa$-complete ultrafilter on $\kappa$ and construct the ultrapower $M = \Ult(V, U)$\footnote{Defining ultrafilters, ultrapowers, and ultraproducts is outside the scope of this paper. For details, see Kanimori \cite{Kanamori:2009}.}. Then the ultrafilter $U$ is in $V_{\kappa+2}$ but not in $M$. In fact, if $j: V \to M$ is elementary and $V_{\kappa+2} \sseq M$, then $\kappa = \crit(j)$ is measurable in $M$, which implies by standard techniques that $\kappa$ is a limit of measurables in $V$. In short, for such elementary embeddings $j$, we always have $V_{\kappa+1} \sseq M$, but to get $V_{\kappa+2} \sseq M$ requires a stronger large cardinal hypothesis. The slogan form of the previous observation is:

\begin{quote}
\emph{When considering elementary embeddings $j: V \to M$: The more that $M$ agrees with $V$, the stronger the corresponding large cardinal axiom.}
\end{quote}

A natural culmination to this hierarchy of axioms would be for the target model $M$ to be all of $V$. In his 1967 thesis William Reinhardt suggested an axiom asserting the existence of such an embedding.

\begin{dfn}[Reinhardt \textup{(NGB)}\footnote{The definition takes place not in ZFC but in von Neumann-G\"odel-Bernays class theory that allows quantification over classes as well as sets.}]
A cardinal $\kappa$ is a \emph{Reinhardt cardinal} \index{Reinhardt cardinal} if $\kappa = \crit(j)$ for some nontrivial elementary embedding $j: V \to V$. 
\end{dfn}

In 1971, Kenneth Kunen showed with the following stronger result that Reinhardt cardinals are inconsistent with the Axiom of Choice. 

\begin{thm}[Kunen \textup{(ZFC)}] \label{Kunen} \index{Kunen's theorem}There is no ordinal $\lambda$ such that there is a non-trivial elementary embedding $j: V_{\lambda+2} \to V_{\lambda+2}$. 
\end{thm}
See Kanamori \cite{Kanamori:2009} for three proofs of Kunen's Theorem, the original form of which asserts that there are no Reinhardt cardinals. (The form we have given here is a corollary.)

Kunen's Theorem gives an upper bound on the possible closure of the target model. Below that upper bound are a number of axioms in which the domain and target models are rank initial segments of the universe $V_{\delta}$ for some ordinal $\delta$.  We will be interested in a family of related large cardinal axioms that assert the existence of embeddings $j:V_{\delta} \rightarrow V_{\delta}$ with various properties. We will begin our discussion with embeddings from $V_{\lambda}$ to $V_{\lambda}$. Such embeddings are called \emph{rank-to-rank (or rank-into-rank) embeddings}; because that name is also sometimes used for the whole family of related embeddings, we note that these are also called I3 embeddings. Recall that $\lambda$ is a limit ordinal of cofinality $\omega$ and that all elementary embeddings are assumed to be nontrivial, unless otherwise stated. (See the end of this section for more on this family of large cardinal axioms.)

Many interesting properties of rank-to-rank embeddings are a consequence of the basic fact that any elementary embedding $j: V_\lambda \to V_\lambda$ can be naturally extended to an elementary embedding on the power set of $V_\lambda$ (that is, to $\vlm$): For $A \sseq V_\lambda$ and $j: V_\lambda \to V_\lambda$ an elementary embedding, define $$j\cdot A = j A = j(A) = \bigcup_{\alpha < \lambda} j(A \cap V_\alpha).$$

The observation that every rank-to-rank embedding $j$ is itself a subset of $V_\lambda$ 
(when encoded as usual for functions as the set of ordered pairs $\la x,j(x)\ra$) allows us to define the application of $j$ to itself, 
or indeed to any other rank-to-rank embedding $k\from V_\lambda\to V_\lambda$. 
The result $jk$ will be another rank-to-rank embedding: 
see Proposition \ref{basicsprop}.\ref{jkelem} below.

Recall
from Section~\ref{FreeLDs} that we parenthesize on the left and omit the $\cdot$ symbol for the application operation: $jk\ell$ denotes $(j\cdot k)\cdot\ell$.

\begin{dfn}\label{defncritseq}
    The \emph{critical sequence} $\overrightarrow{\crit} (j)$ of a rank-into-rank embedding $j: V_\lambda \to V_\lambda$ is defined as the sequence $\la \kappa_n|\, n \in \omega \ra$ such that $\kappa_0 = \crit j$ and $\kappa_{n+1} = j(\kappa_n)$ for all $n < \omega$. 
\end{dfn}

For $j \in \mathcal{E}_{\lambda}$, let $f(n)$ be the number of critical points of members of $\mathcal{A}_j$ between $\kappa_n$ and $\kappa_{n+1}$: $$f(n) = |\{\gamma \in \crit \mathcal{A}_j: \kappa_n < \gamma < \kappa_{n+1}\}|.$$ Perhaps surprisingly, $f(n)$ grows very rapidly: Dougherty showed \cite{Dougherty:1993}
that $f(n)$ dominates the Ackermann function, a benchmark function that grows far faster than exponential or hyperexponential. 
There are no critical points of $\mathcal{A}_j$ between $\kappa_0$ and $\kappa_1$ or between $\kappa_1$ and $\kappa_2$ ($f(0) = 0$ and $f(1) = 0$); there is exactly one between $\kappa_2$ and $\kappa_3$, and $f(3)$ is unimaginably large. This fast growing behavior left open the possibility that there were infinitely many members of the critical family $crit \mathcal{A}_j$ between two elements of the critical sequence of $f$. Combined results of Steel and Laver show that does not happen.

\begin{thm}[The Laver-Steel Theorem]
    For all $n < \omega$, $f(n)$ is finite.
        \begin{itemize}
            \item (Steel \cite{Steel:1993}) For any sequence $\la k_n: n < \omega \ra$ with $k_n\in\mathcal{E}_{\lambda}$, 
            if we define $\ell_n = (((k_0k_1)k_2) \cdots k_{n-1})k_n$, then $\sup \{\crit \ell_n: n< \omega\} = \lambda$.
            \item (Laver \cite{Laver:1995}) The order type of $\{\crit k: k \in \mathcal{A}_j\}$ is $\omega$.
        \end{itemize}
\end{thm}

In our inverse limit construction in Section \ref{Main}, the critical sequence of the given embedding $j \in \Emb_{\lambda}$ will play a vital role. We now present some useful observations about the critical sequence of $j$.

\begin{dfn}\label{leftpowers}
    For any natural number $n$, define $j^{(n)}$, the $n^{th}$ \emph{right power} of $j$, recursively on $n$ by $$j^{(0)} = j \text { and } j^{(n+1)} = j^{(n)} j^{(n)}.$$
\end{dfn}

Note that a simple induction proves that $j^{(n+1)}=j j^{(n)}$ for all $n$. It follows immediately that the critical sequence of $j$ is precisely the sequence of critical points of the right powers of $j$:
$$\crit j^{(n)} = \kappa_n.$$
Indeed, a straightforward induction shows that the behaviour of the right powers $j^{(n)}$ on the critical sequence of $j$ is very simple:

\begin{lem}\label{jmsonkappans}
Let $j$ be in $\Ecal_{\lambda}$ with critical sequence $\overrightarrow{\crit}(j) = \la\ka_n\st n\in\omega\ra$.
For all $m$ and $n$ in $\omega$, 
\[
j^{(m)}(\ka_n)=\begin{cases}
\ka_n&\textrm{if }n<m\\
\ka_{n+1}&\textrm{if }m\leq n.
\end{cases}
\]
\end{lem}
\begin{proof}
By definition of the critical sequence, when $m = 0$, $j^{(0)}(\ka_n) = j(\ka_n) = \ka_{n+1}$ for all $n \geq 0$. Recall that the critical point of $j^{(m)}$ is $\ka_m$ and that the critical points are increasing, so $n<m$ implies that $\ka_n < \ka_m$ and thus $\ka_n$, being below the critical point of $j^{(m)}$, is unmoved by $j^{(m)}$.

Now, assuming inductively that for some $m$ we have
$j^{(m)}(\ka_\ell)=\ka_{\ell+1}$ for all $\ell\geq m$,
it follows from the elementarity of $j$ that $j^{(m+1)}(j(\ka_\ell)) = (jj^{(m)})(j(\ka_\ell)) = j(j^{(m)}(\ka_{\ell})) =j(\ka_{\ell+1})=\ka_{\ell+2}$.
Substituting $n=\ell+1$ we have that $j^{(m+1)}(\ka_n)=\ka_{n+1}$ for all $n\geq m+1$, as required.
\end{proof}

The behavior below $\lambda$ of the right powers of $j$ on the critical sequence of $j$ can be visualized in Figure \ref{rankpicture}.  \index{rank-into-rank embeddings, picture}

\setlength{\unitlength}{1cm}
\begin{figure} \label{rankintorank}
\begin{picture}(1,5.5)
\put(1,0){\line(0,1){6}}
\put(.9,1){\line(1,0){.2}}
\put(.5, .9){$\kappa_0$}
\put(.9,5.5){\line(1,0){.2}}
\put(.5,5.4){$\lambda$}

\put(1.4, 5.5){\vector(1,0){2}}
\put(1.4, 1){\vector(3, 1){2}}
\put(1.4, 0.5){\vector(1,0){2}}
\put(2.3, 3.2){$j$}

\put(4,0){\line(0,1){6}}
\put(3.9,2){\line(1,0){.2}}
\put(2.3,1.9){$j(\kappa_0) = \kappa_1$}
\put(3.9,5.5){\line(1,0){.2}}

\put(4.4, 5.5){\vector(1,0){2}}
\put(4.4, 2){\vector(3, 1){2}}
\put(4.4, 1.5){\vector(1,0){2}}
\put(5.3,3.95){$jj$}

\put(7,0){\line(0,1){6}}
\put(6.9,3){\line(1,0){.2}}
\put(5.3,2.9){$j(\kappa_1) = \kappa_2$}
\put(6.9,5.5){\line(1,0){.2}}

\put(7.4, 5.5){\vector(1,0){2}}
\put(7.4, 3){\vector(3, 1){2}}
\put(7.4, 2.5){\vector(1,0){2}}
\put(10.25, 4.5){$\cdots$}

\put(12,0){\line(0,1){6}}
\put(11.9,5.5){\line(1,0){.2}}
\put(12.3,5.4){$\lambda$}
\end{picture}
\caption{Right powers of rank-into-rank embeddings.}
\label{rankpicture}
\
\end{figure}
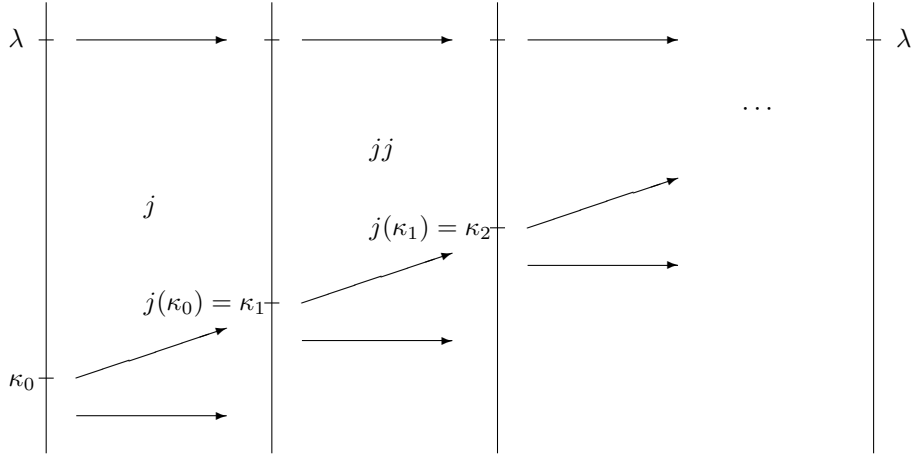

We collect a number of basic properties of rank-to-rank embeddings in the following proposition. 

\begin{prop}\label{basicsprop} Suppose that $j \in \Ecal_{\lambda}$. Then the following hold. \index{rank-into-rank embeddings, properties of}
\begin{enumerate}
\item\label{jaisja} For elements $a$ of $V_\lambda$, $j\cdot a = j(a)$.
\item\label{jkelem} If $k$ is in $\Emb_{\lambda}$, then $jk: V_\lambda \to V_\lambda$ is in $\Emb_{\lambda}$.  In particular, every right power $j^{(n)}$ of $j$ is in $\Emb_{\lambda}$. 
\item\label{critscof} The critical sequence of the right powers of $j$, $\la \crit(j^{(n)}) |\, n < \omega \ra$, is increasing and cofinal in $\lambda$.  \index{critical sequence}
\item\label{coflaom} The cofinality of $\lambda$ is $\omega$.
\end{enumerate}
\end{prop}

\begin{proof} For (\ref{jaisja}) note that because $\lambda$ is a limit, $a \in V_\lambda$ means $a \in V_\beta$ for some $\beta < \lambda$. Therefore, $$j \cdot a  = \bigcup_{\alpha < \lambda} j(a \cap V_{\alpha}) = \bigcup_{\alpha \leq \beta < \lambda} j(a \cap V_{\alpha}) = j(a).$$

To prove (\ref{jkelem}), it suffices to show that, for $j$ and $k$ in $\Ecal_{\lambda}$, $jk$ is in $\Ecal_{\lambda}$ (the rest follows by induction). 
Let $\beta < \lambda$ and let $\varphi(x_1, \ldots, x_n)$ be a formula. We have that $k \rest V_\beta$ is in $V_\lambda$. Denote $k\rest V_\beta$ by $\ell$ for clarity. For each formula in the language of set theory,
the elementarity of the fragment $\ell$ of $k$ for that formula may be expressed as 
$$V_\lambda \models \forall a_1, \ldots, a_n \in V_\beta (\varphi[a_1, \ldots, a_n] \leftrightarrow \varphi[\ell(a_1), \ldots, \ell(a_n)]).$$
Applying $j$ to this statement we have 
$$V_\lambda \models \forall a_1, \ldots, a_n \in V_{j(\beta)} (\varphi[a_1, \ldots, a_n] \leftrightarrow \varphi[j(\ell)(a_1), \ldots, j(\ell)(a_n)]).$$
Since $j(\beta) \ge \beta$, $jk$ restricts to $j(\ell)$, and $\beta < \lambda$ and $\varphi$ were arbitrary, we get that $jk: V_\lambda \to V_\lambda$ is elementary as desired.

To see that (\ref{critscof}) holds, suppose that $\sup \la \crit(j^{(n)})|\, n < \omega \ra = \bar \lambda < \lambda$. It then follows that
$$j(\bar \lambda) = \sup \la \crit(j^{(n+1)})\, n < \omega \ra = \bar \lambda.$$
If this were the case, we would have that $j \rest V_{\bar \lambda+2}: V_{\bar \lambda+2} \to V_{\bar \lambda+2}$ is a nontrivial elementary embedding, contradicting Kunen's theorem.

Part (\ref{coflaom}) then follows from part (\ref{critscof}). 
\end{proof}


We saw in Lemma \ref{jmsonkappans} that each $j^{(m)}$ agrees with $j$ on a tail of 
the sequence $\la\kappa_n\st n\in\omega\ra$. 
A similar argument from the elementarity of $j$ gives the basic fact in Proposition \ref{squareagreement}.
For this proposition 
and hereafter we use the notation that for $j$ in $\Ecal$, the extended range of $j$, $\rnglpo j$, is the range of $j$ when
construed as a function from $V_{\lambda+1}$ to $V_{\lambda+1}$; that is, 
$$\rnglpo j = \{ A \in \vlm |\, \exists B \in \vlm (jB = A)\}.$$
In particular, $\rng j=\rnglpo j \cap V_\lambda$.

\begin{prop} \label{squareagreement} Suppose $j$ is in $\Ecal_{\lambda}$. For every $A$ in $\rnglpo j$, the extended range of $j$, $j$ agrees with its square $jj$ on $A$: $$jA = jjA.$$ \index{$j$, $j[j]$ agreement}
\end{prop}

\begin{proof} Suppose $A$ is in $\rng^+ j$. Then there exists $\bar A$ in $\vlm$ is such that $j\bar A = A$. Let $\beta < \lambda$. 
Then $A \cap V_{j(\beta)}=j(\bar A \cap V_\beta)$. Note that we can write this as 
\[
A \cap V_{j(\beta)}
=j \rest V_{\beta+1} ( \bar A \cap V_\beta).
\]
Applying $j$ to this statement we have
\begin{align*}
j(A \cap V_{j(\beta)})&=
j(j\rest V_{\beta +1} (\bar A \cap V_\beta))\\
&= j(j \rest V_{\beta+1})j(\bar A \cap V_\beta) = j(j \rest V_{\beta+1}) (A \cap V_{j(\beta)})\\ 
&= jj\restr V_{j(\be)+1}(A \cap V_{j(\beta)}) = jj(A \cap V_{j(\beta)}).
\end{align*}
Hence $jj(A \cap V_{j(\beta)}) = jA \cap V_{j(j(\beta))}$. Since $j(j(\beta)) \ge j(\beta) \ge \beta$ and $\beta$ was arbitrary, the proposition follows. 
\end{proof}

The main results of this paper require large cardinal axioms from the wider pantheon of very large large cardinals (see, for example, Kanamori \cite{Kanamori:2009} ). These embeddings preserve formulas of higher syntactic complexity. 
Recall first the Levy hierarchy of formulas for first order languages $\mathcal{L}$ extending $\mathcal{L}_{\epsilon} = \{ \epsilon \}$: a formula $\varphi$ is $\Sigma_0$ or (synonymously) $\Pi_0$ 
if it has no unbounded quantifiers;
that is, if $\varphi$ has only bounded quantifiers, of the forms $\exists v \in w$ and $\forall v \in w$. 
A formula is $\Sigma_{n+1}$ if it is of the form $\exists v_1 \ldots v_k \psi$ for a $\Pi_n$ formula $\psi$, and a formula is $\Pi_{n+1}$ if it is of the form $\forall v_1 \ldots \forall v_k \varphi$ for a $\Sigma_n$ formula $\varphi$.
In the first order setting, variables (and quantifications) are taken to range over \emph{elements} of the structure in question.
However, we will also need to work with second order variables, which we shall use capital letters $A$, $B$, $X$, $Y$, and so forth, 
denoting \emph{subsets} of the structures in question.  
A formula is said to be $\Sigma^1_0$ or (synonymously) $\Pi^1_0$ if it has no quantifications of second order variables.
A formula is $\Sigma^1_{n+1}$ if it is of the form $\exists X_1\ldots\exists X_k\psi$ where $\psi$ is $\Pi^1_n$, and it is
$\Pi^1_{n+1}$ if it is of the form $\forall X_1\ldots\forall X_k\varphi$ where $\varphi$ is $\Sigma^1_n$.

Say that an elementary embedding $j$ from $\mathcal{M}$ to $\mathcal{N}$ is \emph{$\Sigma^1_n$-elementary} if and only if 
for any $\Sigma^1_n$ formula $\varphi$,
$$\mathcal{M} \models \varphi[x_1, \ldots x_k] \text{ if and only if } \mathcal{N} \models \varphi[jx_1, \ldots jx_k].$$
Following Laver \cite{Laver:1997}, we shall also abbreviate this as simply being ``$\Sigma^1_n$,'' even though in other settings
that would normally be used to refer to definition complexity --- we shall never need to refer to definition complexity of 
embeddings outside of the present sentence.

Roughly speaking, embeddings that preserve more complex formulas are stronger embeddings. The axioms below are organized in increasing consistency strength (with more to say about I2 below).

\begin{description}
\item[Axiom I3:] There exists a nontrivial elementary embedding $j: V_\lambda \to V_\lambda$:\\ $\Emb_{\lambda} \neq \emptyset$.
\item[Axiom I2:] There exists a nontrivial elementary embedding $j: V \to M$ for some transitive class $M$ such that the supremum of the critical sequence of $j$, $cr^{\omega}(j) = \sup \overrightarrow{cr} j$, is $\lambda$ and $V_{\lambda} \subseteq M$.
\item[$\Sigma_n^1$-elementary rank-to-rank] There exists a nontrivial $\Sigma_n^1$-elementary embedding $j: V_\lambda \to V_\lambda$.
\item[Axiom I1:] There exists a nontrivial elementary embedding $j: \vlm \to \vlm$. 
\item[Axiom I0:] There exists a nontrivial elementary embedding $j: L(\vlm) \to L(\vlm)$.
\end{description}

We will not further discuss embeddings at the level of I0 or above in this paper. 

The following theorem combines results of Gaifman, Martin, and Powell. Further equivalents can be found in Laver \cite[Theorem 1.3]{Laver:1997}.
\begin{thm} Suppose $j \in \Emb_{\lambda}$. The following are equivalent:
\begin{enumerate}
\item $j$ is $\Sigma_1^1$.
\item  $j$ extends to an elementary embedding $j:V \to M$, $M$ a transitive class with $\vlm \subset M$.
\item $j$ is $\Sigma^1_2$.
\end{enumerate}
In particular, Axiom I2 holds if and only if there is a nontrivial $\Sigma^1_1$-elementary embedding $j\from V_\lambda\to V_\lambda$.
\end{thm}

In \cite{Laver:1997} Laver proves a number of theorems about preservation of $\Sigma^1_n$-elementarity under inverse limits and square roots. We include several of them in the next section and one here.

\begin{thm}[Laver, \cite{Laver:1997}]\label{hkandhcompk}
 If $h$ and $k$ in $\Ecal_{\lambda}$ are $\Sigma_n^1$, then $h \circ k$ and $hk$ are also $\Sigma^1_n$-elementary embeddings in $\Ecal_{\lambda}$.
\end{thm}

We also note here a result of Martin, as cited by Laver.
\begin{thm}[{Martin, see \cite[Theorem 2.3]{Laver:1997}}]\label{uponefree}
    If $n$ is odd and $j$ in $\Ecal_{\lambda}$ is $\Sigma^1_n$, then $j$ is $\Sigma^1_{n+1}$.
\end{thm}

\section{Technical Tools: Inverse Limits and Square Roots}\label{Tools}

In this section we introduce two technical tools used in our construction (see Section \ref{Main}) of embeddings $K$ and $L$ that are suitably independent of one another.  Neither $K$ nor $L$ will lie in the algebra $\Acal_{j}$ generated by our starting rank-to-rank embedding $j$. 
Rather, the embeddings we construct will be inverse limits, defined
following Laver \cite[Section 3]{Laver:1997}.

\begin{dfn}
Suppose we have a sequence $\vec{k}=\la k_n|\, n < \omega \ra$ of embeddings in $\Emb_{\lambda}$. 
Define 
$\bar \lambda_{\vec{k}}$ to be the cardinal $\liminf_{n< \omega} \crit k_n.$
Then for every $a$ in $V_{\bar \lambda_{\vec{k}}}$, set
$$K(a) = (k_0 \circ k_1 \circ \cdots \circ k_n) (a),$$
where $n$ is large enough that $k_m(a) = a$ holds for every $m>n$. 
We call the function
$K = k_0 \circ k_1 \circ \cdots$ so defined the \emph{inverse limit} of $\vec{k}$.  
\end{dfn}

As Laver \cite[Section 3]{Laver:1997} notes, there are two possibilities. Either the $\liminf$ of the critical points $\bar\lambda_{\vec{k}}$ is attained, 
in which case the inverse limit is
essentially just a finite composite, or it is the cofinality $\omega$ supremum of an increasing sequence of the critical
points.  In the latter case, by taking finite composites of blocks of the embeddings $k_i$
if need be, we may always assume that the
sequence of critical points is strictly increasing. These latter inverse embeddings are called \emph{nontrivial inverse limits}, and we will assume that all inverse limits in this paper are nontrivial with $\bar\lambda_{\vec{k}} = \lambda$.

\begin{prop}[Laver {\cite[Lemma 3.1 and Theorem 3.3]{Laver:1997}}] \label{invlimelem} Let $m$ be a natural number.  Suppose $\vec{k}= \la k_n|\, n < \omega \ra$ is a sequence of $\Sigma^1_m$-elementary embeddings from $V_\lambda$ to $V_\lambda$, and let \mbox{$K = k_0 \circ k_1 \circ \cdots$} be the inverse limit of $\vec{k}$. Then $K$ is an a $\Sigma^1_m$-elementary embedding from $V_{\bar \lambda_{\vec{k}}}$ to $V_\lambda$. In particular, if $\bar \lambda_{\vec{k}} = \lambda$ then $K$ is a $\Sigma^1_m$ rank-into-rank embedding from $V_\lambda$ to $V_\lambda$. 
\end{prop}

We note, still following Laver \cite{Laver:1997}, that one can think of inverse limits as a special case of direct limits of a specific form. In particular, suppose for $\la j_i: i < \omega \ra$ (each $j_i \in \Ecal_\lambda$ with $\liminf_{n< \omega} \crit j_i = \bar{\lambda}$), our inverse limit is $J = j_0 \circ j_1 \circ j_2 \circ \cdots$ and $J: V_{\bar{\lambda}} \to V_\lambda$. 
Then one can define the direct limit $\tilde{J} = \cdots \circ \tilde{j}_2 \circ \tilde{j}_1 \circ \tilde {j}_0$, where for each $i$, $\tilde{j}_i = j_0(j_1(j_2( \cdots (j_{i-1}j_i))))$, and $J = \tilde{J} \rest V_{\bar{\lambda}}$,
by repeated use of the identities $k_1\of k_2=k_1k_2\of k_1$ and $(k_1\of k_2)k_3=k_1(k_2k_3)$. 
See Figure \ref{invlimpicture}.

\begin{figure}
\setlength{\unitlength}{1cm}
\begin{picture}(1,5.5)
\put(1,0){\line(0,1){9}}
\put(.9,1){\line(1,0){.2}}
\put(-.2, .9){$\text{crit}(j_0)$}
\put(.9,1.5){\line(1,0){.2}}
\put(-.2,1.4){$\text{crit}(j_1)$}
\put(.9,2){\line(1,0){.2}}
\put(-.2,1.9){$\text{crit}(j_2)$}
\put(.6, 2.5){$\vdots$}
\put(.9, 3.2){\line(1,0){.2}}
\put(.5, 3.1){$\bar \lambda$}
\put(.9,8.5){\line(1,0){.2}}
\put(.5,8.4){$\lambda$}

\put(1.4, 8.5){\vector(1,0){2}}
\put(1.4, 1.1){\vector(3, 2){2}}
\put(1.4, 3.6){\vector(3,2){2}}
\put(2.3, 1.2){$j_0$}

\put(4,0){\line(0,1){9}}
\put(3.9,3){\line(1,0){.2}}
\put(2.2, 2.9){$\text{crit}(j_0(j_0))$}
\put(3.9,3.5){\line(1,0){.2}}
\put(2.2,3.4){$\text{crit}(j_0(j_1))$}
\put(3.9,4){\line(1,0){.2}}
\put(2.2,3.9){$\text{crit}(j_0(j_2))$}
\put(3.6, 4.5){$\vdots$}
\put(3.9, 5.4){\line(1,0){.2}}
\put(2.9, 5.3){$j_0(\bar \lambda)$}
\put(3.9,8.5){\line(1,0){.2}}

\put(4.4, 8.5){\vector(1,0){3}}
\put(4.4, 3.5){\vector(2, 1){3}}
\put(4.4, 5.6){\vector(2,1){2.8}}
\put(4.4, 3){\vector(1,0){3}}
\put(5.8,3.5){$j_0(j_1)$}

\put(8,0){\line(0,1){9}}
\put(7.9,5.5){\line(1,0){.2}}
\put(5.5,5.4){$\text{crit}(j_0(j_1(j_1)))$}
\put(7.9,6){\line(1,0){.2}}
\put(5.5,5.9){$\text{crit}(j_0(j_1(j_2)))$}
\put(7.6, 6.5){$\vdots$}
\put(7.9, 7.2){\line(1,0){.2}}
\put(6.3, 7.1){$j_0(j_1(\bar \lambda))$}
\put(7.9,8.5){\line(1,0){.2}}

\put(8.4, 8.5){\vector(1,0){1}}
\put(8.4, 6){\vector(2, 1){1}}
\put(8.4, 7.2){\vector(2,1){1}}
\put(8.4, 5.5){\vector(1,0){1}}
\put(8.5,5){$j_0(j_1(j_2))$}

\put(10.25, 7){$\cdots$}

\put(12,0){\line(0,1){9}}
\put(11.9,8.5){\line(1,0){.2}}
\put(12.3,8.4){$\lambda$}
\end{picture}
\caption{Direct limit decomposition of an inverse limit.}
\label{invlimpicture}
\end{figure}
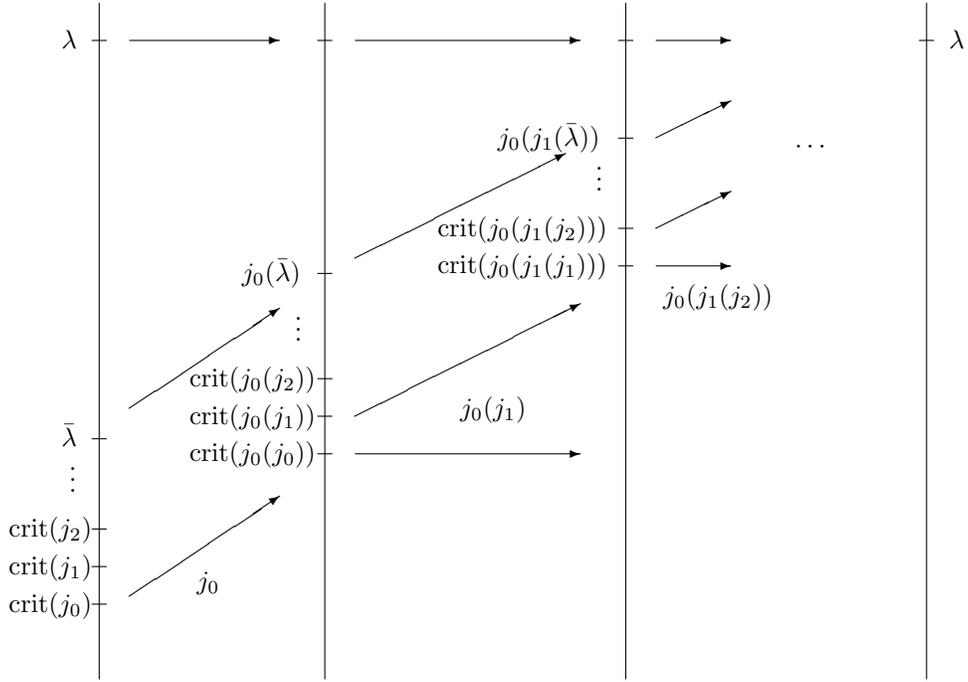

Our next technical tool is the notion of a square root.

\begin{dfn} Suppose that $j$ and $k$ are in $\Ecal_{\lambda}$. We say that $k$ is a  \emph{square root} of $j$ if $kk = j$. 
\end{dfn}

Our first reason for using $\Sigma^1_1$-elementary embeddings is that they are guaranteed to have square roots,
as we now show.  The following is a small extension of the construction implicit in \cite[Lemma~2.2]{Laver:1997};
unfortunately Laver's explanation there is particularly terse.  For those wishing to dig deeper,
Dimonte, Iannella and L\"ucke \cite{DIL:2023} provide a good overview of 
the tree representation of $\Sigma^1_1$-elementarity, 
but our hands-on presentation here was inspired by that given by Kanamori \cite[Proposition 24.4(b)]{Kanamori:2009} 
for the proof that an I2 embedding gives many rank-to-rank embeddings.

\begin{lem}\label{sqrrtlem} Suppose that $j \in \Ecal_{\lambda}$ is a $\Sigma^1_{1}$-embedding. Then for any 
$a_0,\ldots,a_m,$ $b_0, \ldots, b_n \in \vlm$ there is $k \in \Ecal_\lambda$ that is a square root of $j$ such that 
$a_0,\ldots,a_m\in\rnglpo k$ and for all $\ell\leq n$, $k(b_\ell)=j(b_\ell)$.
\end{lem}

\begin{proof}
Note first that by Proposition~\ref{squareagreement}, to ensure that $kb_\ell = kkb_\ell$ for all $\ell\leq n$,
it suffices to show that each $b_\ell$ is in $\rnglpo k$.  Thus, there is no need to differentiate between the
$a_\ell$ parameters and the $b_\ell$ parameters, and we shall henceforth omit any mention of the $b_\ell$ parameters.

The idea of the proof is the fairly common one in large cardinal implication arguments, 
of ``pulling back via $j$'' the existence of
an embedding with properties witnessed by $j$ itself.  As an embedding from $V_\lambda$ to $V_\lambda$, however,
$j$ is too complex an object to make this ``pulling back'' argument directly.  
Thus, we consider a tree of partial embeddings.

So, let $\la\ka_p\st p\in\omega\ra$ be the critical sequence of $j$,  
and define $T$ be the set of functions $i$ such that for some $p\in\omega$,
\begin{itemize}
    \item $i$ is an elementary embedding from $V_{\ka_{p+1}+1}$ to $V_{\ka_{p+2}+1}$,
    \item $i(i\restr V_{\ka_p})=j\restr V_{\ka_{p+1}}$, and
    \item for all $\ell\leq m$, $a_\ell\cap V_{\ka_{p+2}}\in\rng i$.
\end{itemize} 

For all $i$ and $i'$ in $T$, say that $i<i'$ if $i$ extends $i'$ as a function,
that is, $i\supset i'$.
Clearly $T$ is a tree with this partial ordering, and if $T$ has an infinite branch $B$, then defining
$k =\bigcup B$, we have that $k$ is
an elementary embedding from $V_\lambda$ to $V_\lambda$ (as $V_\lambda$ is the direct limit of 
elementary substructures $V_{\ka_p}$) with $a_\ell\in\rnglpo k$ for every $\ell\leq m$, and such that
$kk=j$.  That is, $k$ will be the required square root.

We claim that $T$ indeed has an infinite branch---that is, $T$ is ill-founded.  For consider $j(T)$;
its members are elementary embeddings $i\from V_{\ka_{p+2}+1}\to V_{\ka_{p+3}+1}$, such that
$i(i\restr V_{\ka_{p+1}}) = j(j\restr V_{\ka_{p+1}})$, 
and for all $\ell\leq m$, $j(a_\ell)\cap V_{\ka_{p+3}}\in\rng i$.
Thus, $j\restr V_{\ka_{p+2}+1}$ is itself in $j(T)$ for every $p\in\omega$,
and hence $j(T)$ has an infinite branch.  The existence of an infinite  branch through $j(T)$ is
a $\Sigma^1_1$ statement in the parameter $j(T)$, and so by $\Sigma^1_1$-elementarity of $j$,
there is an infinite branch through $T$.  So $T$ is indeed ill-founded, and the union of any
infinite branch yields the desired square root.
\end{proof}

We note that the analogous result for $\Sigma^1_m$-elementary embeddings also holds for odd $m>1$, yielding 
a $\Sigma^1_{m-1}$-elementary square root.  The proof is via a more straight-forward ``pulling back of $j$''
argument than the above; see Laver \cite[Lemma 2.2]{Laver:1997} for an outline.

\section{Constructing a free left distributive algebra on two generators}\label{Main}

In this section we use the inverse limit construction to produce non-trivial rank-to-rank embeddings that are independent of one another in the sense that the algebra they generate is isomorphic to the free, two-generated, left distributive algebra. 
Laver \cite{Laver:1992} first showed that the closure of a single non-trivial elementary embedding under application generates the free left distributive algebra on a single generator; our results extend that work.

In what follows, let $j$ be an element of $\mathcal{E}_{\lambda}$ with critical sequence $\langle \kappa_i : i < \omega \rangle$. Call $[\kappa_{2n}, \kappa_{2n+1})$ an \emph{even interval of the critical sequence of $j$} (or \emph{even interval}) and $[\kappa_{2n+1}, \kappa_{2n+2})$ an odd interval of the critical sequence of $j$ (or \emph{odd interval}).

\begin{prop} \label{disjointrngs} 
     Suppose $j \in \Emb_{\lambda}$ has critical sequence $\la \kappa_i:\, i < \omega \ra$, and let $k$ be the inverse limit of the right powers of $j$, namely: 
    $$k = j \circ j^{(1)} \circ j^{(2)} \circ \cdots.$$
    Then, up to $\lambda$, all the ordinals in the range of $k$ are either less than $\kappa_0$ or contained in an odd interval of the critical sequence of $j$, i.e., $$\rng k \cap \lambda \sseq [0,\kappa_0) \cup \bigcup_{n < \omega} [\kappa_{2n+1}, \kappa_{2n+2}).$$
\end{prop}

\begin{proof}
Recall that for any $n$ in $\omega$, the critical point of the $n^{th}$ right power of $j$, $j^{(n)} = j(j^{(n-1)})$, is $\ka_n$. Thus the sequence of critical points of the right powers of $j$ is exactly the critical sequence of $j$. Let $k_n = j^{(n)}$. Then, for the sequence $\vec{k}=\la k_n\st n<\omega\ra$, the critical points of the $k_n$ are increasing and cofinal in $\lambda$ ($\bar \lambda_{\vec{k}} = \lambda$), and $k$ is in $\Ecal_\lambda$.

Consider an arbitrary ordinal $\al<\lambda$, and let $n$ in $\omega$ be least such that $\alpha$ is less than the critical point of $k_n$, namely, $\al<\ka_{n}$. We show that the image of $\alpha$ is either less than $\kappa_0$ or is contained in an odd interval of the critical sequence of $j$.

If $n=0$ then $\alpha$ is smaller than the critical point of $j$ and is unmoved by $k$: $k(\al)=\al$ and $\alpha$ is in $[0,\ka_0)$.

Otherwise, by our choice of $n$, $\ka_{n-1}\leq\al<\ka_n$. 
By definition of the inverse limit together with the fact that the critical points of the $j^{(n)}$ are increasing
with $\crit j^{(n)}=\ka_n$, $k(\al)= j\of j^{(1)}\of \cdots\of j^{(n-1)}(\al)$.

Since $\ka_{n-1}\leq\al<\ka_n$, elementarity gives the following relation: $$j\of j^{(1)}\of \cdots\of j^{(n-1)}(\ka_{n-1}) \leq k(\al)< j\of j^{(1)}\of \cdots\of j^{(n-1)}(\ka_n).$$  

Applying Lemma \ref{jmsonkappans} to both the outer expressions $n$ times each, we have 
that $j \circ j^{(1)} \circ \cdots \circ j^{(n-1)}(\kappa_{n-1})=\ka_{2n-1}$
and $j \circ j^{(1)} \circ \cdots \circ j^{(n-1)}(\kappa_{n})=\ka_{2n}$.
Together these give that 
$\ka_{2n-1}\leq k(\al)<\ka_{2n},$ and in particular, $k(\al)$ lies in an odd interval of the critical sequence of $j$, as claimed.
\end{proof}

If we omit the initial ``$j\of$'' from the deifnition of $k$, we get another embedding $\ell$ with all ordinals in its range either being less than $\ka_0$ or
in an \emph{even} interval; thus, the intersections of the ranges of $k$ and $\ell$ will be disjoint above $\ka_0$.
We now want to show that we can get an even stronger dissociation between embeddings. To do this, we discuss the set of all elements in $V_{\lambda}$ definable from parameters from particular sets. 

Recall that an element $b\in M$ is definable from the set $A$ in the $\mathcal{L}$-structure $\mathcal{M}$ with underlying set $M$
if and only if there are a first-order formula $\varphi[x, y_1, \ldots, y_m]$ in the language $\calL$ 
and $a_1, \ldots, a_m \in A$ such that $b$ is the unique element of $M$ for which $M \models \varphi[b, a_1, \ldots, a_m]$ holds.

\begin{dfn}
We denote the collection of all elements of $V_{\lambda}$ definable from the set $A$ 
by $\Def^{V_\lambda}(A)$.
That is,
\[
\Def^{V_\lambda}(A) 
= \{ b \in V_\lambda|\, b \text{ is definable in $V_\lambda$ from elements of $A$}\}.
\]
\end{dfn}

The next proposition shows that for the embedding $k$ defined in Proposition \ref{disjointrngs}, 
no element in an even interval $[\ka_{2n},\ka_{2n+1})$ of the critical sequence of $j$ is definable from the range of $k$, 
even if we allow all the ordinals up to $\ka_{2n}$ as parameters.

\begin{prop} \label{dissprop} 
    Suppose $j$ in $\Ecal_{\lambda}$ is $\Sigma^1_{1}$-elementary with critical sequence $\la \kappa_i|\, i < \omega \ra$, and define 
    $$k = j \circ j^{(1)} \circ j^{(2)} \circ \cdots.$$
    Then for every $n < \omega$, the $n^{th}$ even interval of the critical sequence of $j$ is disjoint from the elements of $V_{\lambda}$ definable from the range of $k$ and sets of rank less than $\ka_{2n}$:
    $$[\kappa_{2n}, \kappa_{2n+1}) \cap \Def^{\,V_\lambda} (\rng k \cup V_{\kappa_{2n}})  = \emptyset.$$
\end{prop}

\begin{proof} 
Consider first the $n = 0$ case, in which we wish to show that $$[\kappa_{0}, \kappa_{1}) \cap \Def^{V_\lambda} (\rng k \cup V_{\kappa_{0}}) = \emptyset.$$  Note that $V_{\kappa_0} \sseq \rng k$ because $k$ is constant on everything below $\ka_0$. The claimed disjointness follows from the fact that anything definable from $\rng k$ is already in $\rng k$:
if $c$ is the unique element of $V_\lambda$ such that $V_\lambda\sat\varphi(c,k(a_0),\ldots,k(a_i))$, then by elementarity of $k$, there is
a $\bar c\in V_\lambda$ such that $V_\lambda\sat\varphi(\bar c,a_0,\ldots,a_i)$, and then $V_\lambda\sat\varphi(k(\bar c),k(a_0),\ldots,k(a_i))$, whence $c=k(\bar c)$. 

Next we consider the case that $n = 1$ and establish the desired disjointness: $$[\kappa_{2}, \kappa_{3}) \cap \Def^{V_\lambda} (\rng k \cup V_{\kappa_{2}}) = \emptyset.$$ Let $b \in \rng k$ and $a \in V_{\kappa_{2}}$. 
By Lemma \ref{sqrrtlem} there is a square root $\bar j: V_\lambda \to V_\lambda$ of $j$ such that the following hold:
\begin{enumerate}
\item $\bar j ( \kappa_0) = \kappa_1$,  $\bar j (\kappa_1) = \kappa_2$, and $\bar j (\kappa_2) = \kappa_3$.
\item $\bar j (j^{-1}(b)) = b$ and $a \in \rng \bar j$. 
\end{enumerate}
Let $\bar k = \bar j \circ j^{(1)} \circ j^{(2)} \circ \cdots.$ Note that we have 
\begin{align*}
\bar k( k^{-1}(b)) & = ( \bar j \circ j^{(1)} \circ j^{(2)} \circ \cdots)( (j \circ j^{(1)} \circ j^{(2)} \circ \cdots)^{-1}(b)) \\
&= ( \bar j \circ j^{(1)} \circ j^{(2)} \circ \cdots)( (j^{(1)} \circ j^{(2)} \circ \cdots)^{-1}(j^{-1}(b)))  \\
& =
  \bar j ( j^{-1}(b)) = b.
 \end{align*}
Also since $a \in V_{\kappa_2}$, we have that 
$\bar j^{-1}(a) \in V_{\kappa_1}$. So the fact that $\crit j^{(1)} \circ j^{(2)} \circ \cdots = \kappa_1$ implies that $\bar j^{-1}(a) \in \rng j^{(1)} \circ j^{(2)} \circ \cdots$. Hence $a \in \rng \bar k$. But notice that $\rng \bar k \cap [\kappa_2, \kappa_3) = \emptyset$, since $\rng j^{(1)} \cap [\kappa_1, \kappa_2) = \emptyset$ implies by elementarity that $\rng \bar j j^{(1)} \cap [\kappa_2, \kappa_3) = \emptyset$ and 
$$
\rng \bar k \sseq \rng (\bar j \circ j^{(1)}) = \rng (\bar j j^{(1)} \circ \bar j) \sseq \rng \bar j j^{(1)}.$$ 
Since $a, b \in \rng \bar k$, it must as before be the case that $\Def^{V_\lambda}(a,b)\subseteq\rng\bar k$, 
and hence $\Def^{V_\lambda} (a, b) \cap [\kappa_2, \kappa_3) = \emptyset$. 

Proving the Proposition for a general $n \ge 1$ is very similar. Suppose $b \in \rng k$ and $ a \in V_{\kappa_{2n}}$, 
and let $\ell=j \circ j^{(1)} \circ \cdots \circ j^{(n-1)}$. 
Using Lemma \ref{sqrrtlem} and Laver's result \cite[Theorem 2.4]{Laver:1997} that each $j^{(i)}$ is also $\Sigma^1_{1}$-elementary, 
we may take $\hat j$ to be a square root of $\ell$ with 
\begin{align*}
    \hat j(\kappa_{n-1}) &= \ell(\kappa_{n-1}) = \kappa_{2n-1}, \\
    \hat j(\kappa_{n}) &= \ell(\kappa_n) = \kappa_{2n},\\  
    \hat j(\kappa_{n+1}) &= \ell(\kappa_{n+1}) = \kappa_{2n+1},
\end{align*}
and $a,b \in \rng \hat j$ with $\hat j(\ell^{-1}(b)) = b$. We then similarly find 
for 
$$\hat k =  \hat j \circ j^{(n)} \circ j^{(n+1)} \circ \cdots$$
that $a, b \in \rng \hat k$.
From $\rng j^{(n)} \cap [\kappa_{n}, \kappa_{n+1}) = \emptyset$ it follows by elementarity that 
$$\rng \hat j j^{(n)} \cap [\kappa_{2n}, \kappa_{2n+1}) = \emptyset,$$ 
so since 
$$\rng \hat k \sseq \rng (\hat j \circ j^{(n)}) = \rng (\hat j j^{(n)} \circ \hat j) \sseq \rng \hat j j^{(n)},$$ 
we have that $[\kappa_{2n}, \kappa_{2n+1}) \cap \rng \hat k = \emptyset.$
Once again, this implies that 
$$\Def^{V_\lambda} (a, b) \cap [\kappa_{2n}, \kappa_{2n+1}) = \emptyset,$$ 
since if $c$ is unique such that $V_\lambda\sat\varphi(c,a,b)$, 
then by elementarity of $\hat k$, there is a unique $\bar c$ such that
$V_\lambda\sat\varphi(\bar c,\hat{k}^{-1}a,\hat{k}^{-1}b),$
and $c$ must be $\hat{k}\bar c$.
Hence, as $a$ was arbitrary in $V_{\ka_{2n}}$ and $b$ was arbitrary in $\rng k$, the result follows.
\end{proof}

We now introduce a kind of uniform restriction of I3 embeddings, extending the concept $\overset{*}{\cap}$ originally introduced by Laver
\cite{Laver:1995}, which we also review. 

\begin{dfn}
     For $\delta \le \lambda$ and $j \in \Ecal_\lambda$, we denote by $j\starint\delta$ the set
     \[
     \{(a,b)\in V_\delta\times V_\delta\st b\in j(a)\}.
     \]
Thus, $j\starint{\delta}$ may be seen as encoding the function from $V_\delta$ to $V_{\delta+1}$ taking $a$ to $j(a)\cap V_\delta$;
note in particular that this is defined even for those $a\in V_\delta$ with $j(a)\notin V_\delta$, 
unlike $(j\restr V_\delta)\cap (V_\delta\times V_\delta)$.     
We then define $j \overset{u}{\cap} \delta$, the \emph{uniform intersection of $j$ up to $\delta$}, 
to be the sequence of the weak intersections up to $\delta$:
$$j \overset{u}{\cap} \delta = \la j \starint\alpha|\, \alpha < \delta \ra.$$ 
\end{dfn}

Note that for $\delta$ a limit ordinal, $\beta \le \delta$, and $j$ and $k$ in $\Ecal_\lambda$, we have that 
$$(j \starint\delta)(k \overset{u}{\cap} \beta) = j(k) \overset{u}{\cap} \gamma$$
where $\gamma = \min \{ j(\beta), \delta \}$. To see this, we compute 
\begin{align*}
    j \starint\delta(k \overset{u}{\cap} \beta) &= j(\la k \starint\alpha\st\, \alpha < \beta \ra) \cap V_\delta\\
    &= \la j(k) \starint\alpha\st\, \alpha < j(\beta) \ra \cap V_\delta \\
    &=\la j(k) \starint\alpha\st\, \alpha < \gamma \ra.
\end{align*}

Also note that for $\delta$ a limit, from $j \overset{u}{\cap} \delta$ we can define $j \starint\delta$, as
$$(j \starint\delta)(a) = \bigcup_{\alpha < \delta} (j \starint\alpha)(a).$$ 

We now prove a lemma which allows us to propagate the property in Proposition \ref{dissprop} to \emph{iterated} ranges (see Definition \ref{iterrange}). 

\begin{lem} \label{irngempty} Suppose that $j$ in $\Ecal_\lambda$ is an elementary embedding, 
$\delta$ and $\mu$ are 
cardinals such that $\crit j < \delta < \mu < \lambda$, and 
$$\Def^{\,V_\lambda} (\rng j \cup V_\delta) \cap [\delta, \mu) = \emptyset.$$ 
Then for any $k \in \Ecal_\lambda$ 
of the form $k = j j_1 j_2 \cdots j_n$, where $\crit j j_1 \cdots j_i < \delta$ for all $i \le n$, we have that 
$$
\rng k \cap [\delta, \mu) = \emptyset.$$ 
\end{lem}

\begin{proof} 
Note first that since
\begin{itemize}
    \item $\crit j<\delta<\lambda$, 
    \item $\lambda$ is the least closure point of $j$ above $\crit j$, and 
    \item $\rng j$ is disjoint from $[\delta,\mu)$, 
\end{itemize}
there must be some $\mu^* < \delta$ such that $j(\mu^*) \ge \mu$,
as otherwise $\delta$ would be a closure point of $j$. 

We prove the lemma by induction on $n$. (We will describe our induction hypothesis below.) Suppose we are in the set-up of the lemma for $n = 1$. Let $\mu_0 < \delta$ be the least ordinal such that $j(\mu_0) \ge \mu$.  Note that $j(j_1 \starint{\mu_0}) = j j_1 \starint{j(\mu_0)}$. 

Suppose towards a contradiction that $\rng j j_1 \cap [\delta, \mu) \neq \emptyset$. Let $\alpha$ be least such that $(j j_1) (\alpha) \in [\delta, \mu)$. We must have that $\alpha < \delta$, since again, $\crit jj_1 < \delta$ and $\lambda$ is the
least closure point of $jj_1$ greater than $\crit jj_1$. Note that since $j(\mu_0) \ge \mu$, we have that $(j j_1 \starint{j(\mu_0)})(\alpha) = (j j_1)(\alpha)$. 
But $\alpha \in V_\delta$ and $j j_1 \starint{j(\mu_0)} \in \rng j$, hence this contradicts that 
$$\Def^{V_\lambda} (\rng j \cup V_\delta) \cap [\delta, \mu) = \emptyset.$$
So $\rng j j_1 \cap [\delta, \mu) = \emptyset$, confirming the $n=1$ case of the lemma.  

For the sake of the induction to come, let $\mu_1< \delta $ be least  such that $j j_1 (\mu_1) \geq \mu$. 
Note that, since $\mu$ is a limit ordinal, $\mu_1$ will also be a limit ordinal.
Now let $\gamma_1 = \min\{j(\mu_0), j j_1(\mu_1)\}$. We have that $$\gamma_1\in \Def^{V_\lambda} (\rng j \cup V_\delta)$$
since either $\gamma_1 \in \rng j$, or $\gamma_1 = (j j_1 \starint{j(\mu_0)})(\mu_1)$ with $\mu_1$ below $\delta$.  
Likewise, $j j_1 \starint{\gamma_1}$ and $j j_1 \overset{u}{\cap} \gamma_1$ 
also lie in $\Def^{V_\lambda} (\rng j \cup V_\delta)$.

Now assume that for all $1 \le i < n$  we have the following:
\begin{enumerate}
\item $\rng j j_1 j_2 \cdots j_i \cap [\delta, \mu) = \emptyset$.
\item $\mu_i < \delta$ is least such that $(j j_1 j_2 \cdots j_i)(\mu_i) \geq \mu$.
\item $\gamma_i = \min \{ j(\mu_0), j j_1(\mu_1), j j_1 j_2 (\mu_2), \ldots, j j_1 j_2 \cdots j_i(\mu_i)\} \ge \mu.$
\item\label{jj1jirestrinDef} $j j_1 j_2 \cdots j_i \overset{u}{\cap} \gamma_i \in \Def^{V_\lambda} (\rng j \cup V_\delta)$.
\end{enumerate}
Note first that since $\ga_{n-1}$ will be a limit ordinal, (\ref{jj1jirestrinDef}) implies also that 
\[
j j_1 j_2 \cdots j_{n-1} \starint{\gamma_{n-1}}\in\Def^{V_\lambda} (\rng j \cup V_\delta).
\]
Bringing in $j_n$, we have
$$
j j_1 j_2 \cdots j_{n} \overset{u}{\cap} \gamma_{n-1} = (j j_1 j_2 \cdots j_{n-1} \starint{\gamma_{n-1}})(j_{n} \overset{u}{\cap} \mu_{n-1}),
$$
since $\gamma_{n-1}\leq( j j_1 j_2 \cdots j_{n-1})(\mu_{n-1})$. 
Therefore, since  $j_{n} \overset{u}{\cap} \mu_{n-1} \in V_\delta$, we may conclude
that 
$$
j j_1 j_2 \cdots j_{n} \overset{u}{\cap} \gamma_{n-1} \in  \Def^{V_\lambda} (\rng j \cup V_\delta).
$$
Again, since $\gamma_{n-1}$ is a limit, this further implies that 
$j j_1 j_2 \cdots j_{n} \starint{\gamma_{n-1}} \in  \Def^{V_\lambda} (\rng j \cup V_\delta).$

Suppose towards a contradiction that $$\rng j j_1 j_2 \cdots j_{n} \cap [\delta, \mu) \neq \emptyset.$$ Let $\alpha$ be least such that $j j_1 j_2 \cdots j_{n}(\alpha) \in [\delta, \mu)$. Then $\alpha < \delta$ since $\crit j j_1 j_2 \cdots j_{n} < \delta$ and $\lambda$ is the least closure point of $j j_1 j_2\cdots j_{n}$ greater than its 
critical point. Because $\gamma_{n-1} \ge \mu$, we have 
$$(j j_1 j_2 \cdots j_{n} \starint{\gamma_{n-1}}) (\alpha) = (j j_1 j_2 \cdots j_{n})(\alpha),$$
and obtain that 
$$(j j_1 j_2 \cdots j_{n})(\alpha) \in  \Def^{V_\lambda} (\rng j \cup V_\delta).$$
This contradicts the assumption that  $ \Def^{V_\lambda} (\rng j \cup V_\delta) \cap [\delta, \mu) = \emptyset$. 
Therefore, $\rng j j_1 j_2 \cdots j_{n} \cap [\delta, \mu) = \emptyset$, as desired.

We now show that our induction hypothesis continues to hold, in case we want to continue the induction beyond $n$. 
 Let $\mu_{n} < \delta$ be  least  with $$(j j_1 j_2 \cdots j_{n})(\mu_{n}) \ge \mu.$$  Then let
$$\gamma_{n} = \min \{ j(\mu_0), j j_1(\mu_1), j j_1 j_2 (\mu_2), \ldots, j j_1 j_2 \cdots j_{n-1}(\mu_{n-1}), j j_1 j_2 \cdots j_{n}(\mu_{n})\}.$$
Clearly $\gamma_{n} \ge \mu$. 

To see that (\ref{jj1jirestrinDef}) of the induction hypothesis holds, 
notice that either $\gamma_{n} = \gamma_{n-1}$ or $\gamma_{n} =  j j_1 j_2 \cdots j_{n}(\mu_{n})$. 
We saw above that
$$j j_1 j_2 \cdots j_{n} \overset{u}{\cap} \gamma_{n-1} \in \Def^{V_\lambda} (\rng j \cup V_\delta),$$
so 
if $\gamma_{n} = \gamma_{n-1}$ then we are done. 
If on the other hand $\gamma_{n-1} > \gamma_{n} =  j j_1 j_2 \cdots j_{n}(\mu_{n})$, then 
$$(j j_1 j_2 \cdots j_{n} \starint{\gamma_{n-1}}) (\mu_{n}) = \gamma_{n}.$$
This implies that $\gamma_{n} \in \Def^{V_\lambda} (\rng j \cup V_\delta)$, and hence that 
$$j j_1 j_2 \cdots j_{n} \overset{u}{\cap} \gamma_{n} \in \Def^{V_\lambda} (\rng j \cup V_\delta),$$ 
since we can define $j j_1 j_2 \cdots j_{n} \overset{u}{\cap} \gamma_{n}$ from $\gamma_{n}$ and $j j_1 j_2 \cdots j_{n} \overset{u}{\cap} \gamma_{n-1}$.
That completes the induction, and so the proof of the lemma.
\end{proof}

We can now use this lemma to show that there are embeddings whose \emph{iterated} ranges are disjoint. 
\begin{dfn} \label{iterrange}
    For $j \in \Emb_\lambda$, define $\irng (j)$, the \emph{iterated range} of $j$, by induction as follows. $\irng_0 (j) = \rng^+ j \cap \Emb_\lambda$, and for $n < \omega$, 
$$\irng_{n+1} (j) = \bigcup_{k \in \irng_n j} \irng_0 (k).$$ 
And we define
$$\irng (j) = \irng_\omega(j) = \bigcup_{n < \omega} \irng_n(j).$$ 
\end{dfn}

\begin{thm} \label{irngthm} Suppose $j \in \Ecal_\lambda$ is a $\Sigma^1_m$-elementary embedding. Then there are $\Sigma^1_m$-elementary $k$ and $\ell$ in $\Ecal_\lambda$ with disjoint iterated ranges: $\irng (k) \cap \irng (\ell) = \emptyset$. 
\end{thm}

\begin{proof}  Define $k$ and $\ell$ by 
    $$k = j \circ j^{(1)} \circ j^{(2)} \circ \cdots$$
    and 
        $$\ell = j^{(1)} \circ j^{(2)} \circ \cdots.$$ By Proposition \ref{invlimelem}, $k$ and $\ell$ are $\Sigma^1_m$-elementary embeddings in $\Ecal_\lambda$. 
        
        By Proposition \ref{dissprop} we have that for $\la \kappa_n|\, n < \omega \ra$ the critical sequence of $j$, and for any $ n < \omega$, 
$$[\kappa_{2n}, \kappa_{2n+1}) \cap \Def^{V_\lambda} (\rng k \cup V_{\kappa_{2n}}) = \emptyset$$
and 
$$[\kappa_{2n+1}, \kappa_{2n+2}) \cap \Def^{V_\lambda} (\rng \ell \cup V_{\kappa_{2n+1}}) = \emptyset.$$
Hence we can apply Lemma \ref{irngempty} to $k$ to get that for any $k' \in \irng(k)$, if $n$ is large enough (so that the critical point condition is satisfied), then 
$$[\kappa_{2n}, \kappa_{2n+1}) \cap \rng k' = \emptyset.$$ 
Similarly we can apply Lemma \ref{irngempty} to $\ell$ to get that for any $\ell' \in \irng(\ell)$, if $n$ is large enough, then 
$$[\kappa_{2n+1}, \kappa_{2(n+1)}) \cap \rng \ell' = \emptyset.$$ 
But these facts together imply that there can be no common embedding in $\irng(k)$ and $\irng(\ell)$, since the critical sequence of that embedding would be cofinal in $\lambda$ and common to their ranges.
\end{proof}

We can now use this theorem to prove our main result that there are embeddings that generate a free left-distributive algebra on two generators. 




\begin{thm}\label{freealgthm}
Suppose that $k$ and $\ell$ are $\Sigma_{1}^1$-elementary embeddings in $\Ecal_{\lambda}$ with disjoint iterated ranges: $\irng (k) \cap \irng (\ell) = \emptyset$. Then $\mathcal{A}_{k, \ell}$ is isomorphic to $\mathcal{A}_2$, the free left distributive algebra on two generators.
\end{thm}

\begin{proof}
Consider two terms $u$ and $v$ in $A_{k, \ell}$. We need to show that they are equivalent as embeddings if and only if they are equivalent under the left distributive law. The reverse implication is immediate: if they are LD-equivalent, they are equivalent as embeddings.
It then remains to show that if two embeddings in $\Acal_{k, \ell}$ are not equivalent under the left distributive law then they are not equivalent as embeddings.

Recall from Section \ref{FreeLDs} that although $<_L$ is not a linear order on any many-generated free left distributive algebra $\mathcal{A}_{\delta}$ (for cardinal $\delta >1$), the conservative extension $\mathcal{P}_{\delta}$, and consequently also $\Acal_{\delta}$, satisfies quadrichotomy:
for any pair of terms $u$ and $v$ in $\mathcal{P}_{\delta}$, either $u$ and $v$ are equivalent, one of $u$ and $v$ is a left subterm of the other, or they have a variable clash. That is, for all $u$ and $v$ in $\mathcal{A}_{\delta}$, either $u \equiv_{LD} v$, $u <_L v$, $v <_L u$, or $u \nsim v$ 
(see \cite[Theorem 14]{Miller:2014} for a proof and \cite[Chapter V, Section 3]{Dehornoy:2000}  
for the surrounding theory of \emph{confluence}).

Let $x$ and $y$ be the generators of $\mathcal{A}_2$. If $w$ is a term in $P_2$, the 
set of terms in $x$, $y$, and the operations $\cdot$ and $\of$,
let $\overline{w} = w[k,\ell]$ be the homomorphic image of $w$ in $\calP_{k,\ell}$---the closure of $k$ and $\ell$ under application and composition---obtained from the map $\overline{x} = x[k,\ell] \rightarrow k$ and $\overline{y} = y[k, \ell] \rightarrow \ell$.

    If $\mathcal{A}_{k,\ell}$ were not free, there would be a pair $u$, $v$ in $\mathcal{A}_2$ such that $u \not\equiv_{LD} v$ but $\overline{u} = u[k,l] \equiv v[k,l] = \overline{v}$ as embeddings in $\Ecal_{\lambda}$. 
    So assume that $u$ and $v$ in $A_2$ are not LD-equivalent. We will show that $\overline{u}$ and $\overline{v}$ in $\Acal_{k, \ell}$ are not equal as embeddings.

    Because $u \not\equiv_{LD} v$ by assumption, we can apply quadrichotomy to conclude (without loss of generality) that either $u$ is a left subterm of $v$, $u <_L v$, or that the two words have a variable clash, $u \nsim v$ (as elements of the free algebra $\mathcal{A}_2$). We consider the two cases separately.

    Case 1: If $u <_L v$ as elements of $\Acal_2$, then there exist $v_1, \ldots, v_n$ in $\Acal_2$ such that $v \equiv_{LD} uv_1 \cdots v_n$. Therefore, as embeddings, $$ v[k,\ell] = u[k,\ell]v_1[k,\ell] \cdots v_n[k,\ell].$$ This proves that $\overline{u} = u[k,\ell] \neq v[k,\ell] = \overline{v}$, as we would otherwise have $\overline{u} <_L \overline{u}$, contradicting the irreflexivity of $<_L$.

    Case 2: If $u \nsim v$ in $\Acal_2$, then, when viewing $u$ and $v$ as members of $\Pcal_2$, there exists a (possibly empty) string $p$ so that $px \leq_L u$ and $py \leq_L v$ for generators $x$ and $y$, with $x \neq y$. We further note that if these inequalities are strict, there exist $v_1, \ldots, v_n$ and/or $u_1, \ldots, u_m$ such that $u = pxu_1 \cdots u_{m-1} u_m$ and $v = pyv_1 \cdots v_{n-1} v_n$.\footnote{Note that because $u$ and $v$ are in $\Acal_2$, when we write each with $p$ on the left, the operation between each of the subterms $u_i$ and $v_j$ must be application (not composition). For details, see \cite{Laver:1992} or \cite{LaverMiller:2013}.}

    The homomorphic images of the inequalities $px \leq_L u$ and $py \leq_L v$ in $\Acal_{k, \ell}$ are $\overline{p}k \leq_L \overline{u}$ and $\overline{p} \ell \leq_L \overline{v}$.

    Observe that $\overline{p}$ is either a member of $\Acal_{k, \ell}$ or a composition of members of $\Acal_{k, \ell}$.
    Either way, by Theorem \ref{hkandhcompk}, $\overline{p}$ is a $\Sigma^1_1$-elementary embedding in $\Ecal_\lambda$, and hence is $\Sigma^1_2$ by Theorem \ref{uponefree}.  Using the fact that being in $\Emb_\lambda$ is $\Sigma^1_2$ (see the proof of Lemma 2.1  in \cite{Laver:1997}), we have that the statement that $k$ and $\ell$ have disjoint iterated ranges is a $\Pi^1_2$ statement (essentially: $\forall\la\la k_1,\ldots,k_m\ra,\la \ell_1,\ldots,\ell_n\ra\ra$, if all $k_i$ and $\ell_j$ are elementary, then
    $kk_1\cdots k_m\neq\ell\ell_1\cdots\ell_n$), and it is preserved by $\overline{p}$. Hence we have that $\overline{p}k$ and $\overline{p} \ell$ have disjoint iterated ranges. 
    
    If in $\Acal_2$ either $u = px$ or $v = py$, we are done: Suppose (without loss of generality) that $u = px$. Then $\overline{v} = \overline{p} \ell \overline{v_1} \cdots \overline{v_n}$ has disjoint iterated range from $\overline{u} = \overline{p} k$ and therefore cannot be equal to $\overline{u}$.

   Suppose then that $\overline{u} = \overline{p} k \overline{u_1} \cdots \overline{u_m}$ and $\overline{v} = \overline{p} \ell \overline{v_1} \cdots \overline{v_n}$. 

We have that $\overline{p} k \overline{u_1} \cdots \overline{u_m}$ is in the iterated range of $\overline{p}k$ and $\overline{v} = \overline{p} \ell \overline{v_1} \cdots \overline{v_n}$ is in the iterated range of $\overline{p}\ell$. The  iterated ranges of $\overline{p}k$ and $\overline{p}\ell$ are, however, disjoint. Therefore, $\overline{u}$ and $\overline{v}$ cannot be equal as embeddings in $\mathcal{A}_{k,\ell}$.
\end{proof}

The two previous theorems combined can be stated as the following corollary.

\begin{cor} Suppose there exists a $\Sigma^1_1$-elementary embedding in $\Ecal_{\lambda}$. Then there are embeddings $k$ and $\ell$ in $\Ecal_{\lambda}$ such that $\Acal_{k, \ell}$, the algebra generated by the closure of $k$ and $\ell$ under application $\cdot$, is isomorphic to $\Acal_2$, the free left distributive algebra on two generators. 
\end{cor}

\section{Translation into Laver Tables}

In this section we translate the results of Section \ref{Main} into the setting of Laver tables, following the development in \cite{Laver:1995}. In so doing, we are able to derive a consequence that does not \emph{a priori} involve elementary embeddings.

We recall from \cite{Laver:1995} the definition of $A_k = (2^k, *_k)$ for $k < \omega$, with operation $*_k$ recursively
given by
\begin{align*}m *_k 0 &= 0, \\
m *_k 1 &= m+1\pmod{2^k}, \textrm{ and} \\
m *_k i & = [m *_k (i-1)] *_k [m *_k 1] \qquad\textrm{ for } 1 < i < 2^k.
\end{align*}
We call $A_k$ the \emph{$k^{th}$ Laver Table}. Following \cite{Laver:1995}, we define for $j$ in $\Ecal_\lambda$ and $n < \omega$,
$$\Acal_{j,n} = \{k \overset{*}{\cap} V_{\delta_n}|\, k \in \Acal_j\},$$
where $\{\delta_n|\, n < \omega \}$ is an increasing enumeration of $\crit(\Acal_j)$. We take application and composition on $\Acal_{j,n}$ to be the operations induced by those on $\Acal_j$ via the projection map $k \mapsto k \overset{*}{\cap} V_{\delta_n}$. Laver showed that the behavior of the embeddings in $\Acal_j$ on the critical family $\crit(\Acal_j)$ is witnessed by the finite algebras $A_n$. Namely, he proved:

\begin{thm}[Laver \cite{Laver:1995}]
    For any $j$ in $\Ecal_\lambda$ and any $n < \omega$, the algebras $A_n$ and $\Acal_{j,n}$ are isomorphic. 
\end{thm}

For $n>m$, let $\pi_{n,m}$ be the natural projection map $\Acal_{j,n} \to \Acal_{j,m}$. Define, as in \cite{Laver:1995}, the inverse limit of these projection maps to be $\mathcal{D}_j$. We can consider $\mathcal{D}_j$ as the set
$$\mathcal{D}_j = \{k|\, \forall n < \omega( k \cap^* V_{\delta_n} \in \Acal_{j,n})\}.$$
One can similarly define the natural projection maps $\bar \pi_{n,m}$ from $A_n \to A_m$ and the corresponding inverse limit $\mathcal{D}^*$. We then have that $ \mathcal{D}^* \cong \mathcal{D}_j$.

The embeddings $k$ and $\ell$ defined in the proof of Theorem \ref{irngthm} are both in $\mathcal{D}_j$. We therefore have the following corollary of Theorem \ref{freealgthm}, the conclusion to which does not involve elementary embeddings.

\begin{cor}
    Suppose there exists a $\Sigma^1_1$-elementary embedding in $\Ecal_\lambda$. Then there are $u, v \in \mathcal{D}^*$ such that $\Acal_{u,v}$ is isomorphic to $\Acal_2$, the free left distributive algebra on two generators. 
\end{cor}

\bibliographystyle{plain}
\bibliography{BTCMreferences}
\end{document}